\documentclass[14pt]{amsart}
\address{\newline{\normalsize Laboratory of AGHA, Moscow Institute of Physics and Technology, 9 Institutskiy per., Dolgoprudny,
Moscow Region, 141701, Russia}
\newline{\it E-mail address}: karzhemanov.iv@mipt.ru}
\usepackage{amscd,amsthm,amsmath,amssymb}
\usepackage[dvips]{graphicx}
\usepackage[matrix,arrow]{xy}

\makeatletter\@addtoreset{equation}{section}\makeatother
\renewcommand{\theequation}{\thesection.\arabic{equation}}

\renewcommand{\thesubsection}{\bf\thesection.\arabic{equation}}
\makeatletter\@addtoreset{subsection}{equation}\makeatother

\newtheorem{theorem}[equation]{Theorem}
\newtheorem{prop}[equation]{Proposition}
\newtheorem{lemma}[equation]{Lemma}
\newtheorem{cor}[equation]{Corollary}
\newtheorem{conj}[equation]{Conjecture}

\theoremstyle{remark}
\newtheorem{remark}[equation]{Remark}

\newtheorem*{notation}{Conventions}

\theoremstyle{definition}

\newtheorem{ex}[equation]{Example}
\newtheorem{question}{Question}

\newcommand{\com}{\mathbb{C}}
\newcommand{\ra}{\mathbb{Q}}
\newcommand{\f}{\mathbb{F}}
\newcommand{\aut}{\text{Aut}}
\newcommand{\p}{\mathbb{P}}

\newcommand{\map}{\longrightarrow}
\newcommand{\cel}{\mathbb{Z}}

\thanks{{\it MS 2010 classification}: 14E08, 14M20, 14M27}


\title{Around the uniform rationality I}

\author{Ilya Karzhemanov}

\begin{document}

\begin{abstract}
We prove that there exist rational but not uniformly rational
smooth algebraic varieties. The proof is based on computing
certain numerical obstruction developed in the case of
compactifications of affine spaces. We show that for some
particular compactifications this obstruction behaves differently
when compared to the uniformly rational situation.
\end{abstract}

\maketitle

\bigskip

\section{Introduction}
\label{section:intro}

\refstepcounter{equation}
\subsection{}
\label{subsection:intro-0}

Let $X$ be a complex projective manifold of dimension $n\ge 2$.
Recall that rationality of $X$ (i.\,e. the existence of a
birational map $X\dashrightarrow\p^n$) yields a Zariski open
subset $U \subset X$ isomorphic (as an affine scheme) to a domain
in $\com^n$. Any rational $X$ obviously carries a family of
(\emph{very free} in the terminology of e.\,g.
\cite{kollar-rat-curves}) rational curves and one may try to
obtain a family of holomorphic maps $\com\map X$ (called
\emph{sprays} in \cite{gromov-oka-pri}) such that near each of its
points $X$ is \emph{(algebraically) h\,-\,Runge} (see
\cite{gromov-oka-pri} for precise definitions and results). The
main expectation is that Zariski locally near \emph{every} point
$X$ should actually look like as an open subset
$U\subseteq\com^n$. One refers to the latter property as
\emph{uniform rationality} (of $X$), the notion introduced
recently in \cite{bog-boh} (following \cite{gromov-oka-pri}),
where some examples and basic properties of uniformly rational (or
u.\,r. for short) manifolds have been established. The ultimate
goal was to approach the following:

\begin{question}[{cf. \cite{bog-boh}, \cite{gromov-oka-pri}}]
\label{theorem:misha-ur} Is it true that every rational manifold
is uniformly rational?
\end{question}

Note that spherical (e.\,g. toric) varieties and blowups of u.\,r.
varieties at smooth centers are easily seen to be u.\,r. (all this
is contained in \cite{bog-boh} (see also \cite{bodn},
\cite{lie-pti}), together with examples of $X$ being the
intersection of two quadrics, small resolution of a singular cubic
$3$\,-\,fold, and some other instances).\footnote{~We do not treat
here two very interesting questions, both discussed in
\cite{bog-boh}, about the \emph{rectifiability} of divisorial
families and on the local regularization of an arbitrary
birational map $X\dashrightarrow\p^n$. Instead we rather
concentrate on the ``\,negative side\,'' of
Question~\ref{theorem:misha-ur} (see below). In addition, recall
that small resolution of a Lefschetz cubic is (Moishezon and) not
u.\,r., which shows that projectivity assumption on $X$ is crucial
for Question~\ref{theorem:misha-ur} to be of any content.} This
immediately gives positive answer to
Question~\ref{theorem:misha-ur} in the case when $n = 2$. It is
also easy to see that all points on a rational manifold $X$ which
may not admit affine neighborhoods $U\subseteq\com^n$ form a locus
of codimension $\ge 2$ (compare with
Proposition~\ref{theorem:prop-on-mults} below).

The goal of the present paper is to prove the following:

\begin{theorem}
\label{theorem:main} In the previous notation, there exists
rational, but not uniformly rational $X$, whenever $n$ is at least
$4$.
\end{theorem}

Thus Theorem~\ref{theorem:main} answers
Question~\ref{theorem:misha-ur} negatively. But still it would be
interesting to find out whether a sufficiently large power $X^N :=
X\times X\times\ldots$ of any rational manifold $X$ is u.\,r.
(same question for the ``\,stabilization\,'' $X\times\p^N$ of
$X$). More examples and discussion will be provided in the
forthcoming paper \cite{around-u-r-II}.

\refstepcounter{equation}
\subsection{}
\label{subsection:intro-1}

We proceed with a description of the proof of
Theorem~\ref{theorem:main}. First of all, in view of the above
discussion it is reasonable to treat only rational manifolds,
which are ``\,minimal\,'' in certain sense (like those that are
not blowups of other manifolds for instance). The most common ones
are compactifications of \emph{affine} $\com^n$ with $\text{Pic}
\simeq \cel$ (see {\ref{subsection:proo-0}} below for a setup).
Next, one may guess that being u.\,r. for a rational manifold $X$
results in ``\,homogeneity property\,'' for the underlying set of
points (compare with \cite[$3.5. \text{E}'''$]{gromov-oka-pri}).
The latter means (ideally) that an appropriate test function $G:
X\map\mathbb{R}$ (``\,Gibbs distribution\,'') must be constant on
uniformly rational $X$. More precisely, as soon as just the points
on $X$ are concerned, it is natural to look for such $G$ that are
conformly invariant (these constitute a class of the so\,-\,called
\emph{asymptotic invariants} of $X$).

Now, if $L$ is an ample (or, more generally, nef) line bundle on
$X$, then the very first candidate for $G$ one can think of would
be the \emph{Seshadri constant} $s_{\scriptscriptstyle L}(\cdot)$
of $L$. Namely, for any point $o \in X$ let us consider the blowup
$\beta: Y \map X$ of $o$, with exceptional divisor $\mathcal{E} :=
\beta^{-1}(o)$, and then put
$$
G(o) := s_{\scriptscriptstyle L}(o) :=
\max\left\{\lambda\in\mathbb{R}\ \vert \ \text{the divisor}\
\beta^*L - \lambda \mathcal{E}\ \text{is nef}\right\}.
$$
We will also write simply $s(o)$ instead of $s_{\scriptscriptstyle
L}(o)$ when $\text{Pic}\,X = \cel \cdot L$ (resp. when $L$ is
clear from the context). Further, the conformal invariance of $G =
s_{\scriptscriptstyle L}(\cdot)$ may be seen via another
definition of it as follows (see \cite{dem}):
\begin{equation}
\label{sesh-const} s_{\scriptscriptstyle L}(o) := \sup
\frac{\text{mult}_o\,\mathcal{M}}{k},
\end{equation}
where the supremum is taken over all $k\in\cel$ and linear
subsystems $\mathcal{M}\subseteq |kL|$ having \emph{isolated} base
locus near $o$. One of the key ingredients in establishing the
expression \eqref{sesh-const} is the Poincar\'e\,--\,Lelong
formula
$$
\text{mult}_o\,D = \lim_{r\to
0}\frac{1}{r^{2n-2}}\text{Vol}\,(D\cap B(r)) = \lim_{r\to
0}\frac{1}{r^{2n-2}}\int_{D\cap B(r)}\omega^{2n-2}
$$
for the multiplicity of a hypersurface $D \subset X$ at the point
$o$ (here $\text{mult}_o\,D$ is computed in terms of the volumes
of intersections of $D$ with small balls $B(r)$ centered at $o$).

Recall that the basic play ground for our approach are those $X$
containing $U := \com^n$ as a Zariski open subset. We also require
the boundary $\Gamma := X \setminus U$ to be (of pure codimension
$1$ and) \emph{irreducible}. Assuming such $X$ uniformly rational,
we claim that $s(\cdot)$ attains the same value at some points on
$U$ and $\Gamma$, respectively (see
Corollary~\ref{theorem:cor-prop-on-mults} for a precise
statement). One thus gets a relatively simple numerical criterion
to test uniform rationality of the manifolds in question. The main
issue then is to find a particular $X$ for which this obstruction
actually gives something non\,-\,trivial.

For $n = 3$, as we show in \cite{around-u-r-II}, one does not
obtain anything interesting. However, the needed examples, for
\emph{any} $n \ge 4$, are constructed in Section~\ref{section:exa}
below. The idea behind our construction is to mimic the one for
the $4$\,-\,fold $V_5^4$ from \cite{pro-coindex-3}. Namely, we
start by blowing up $\p^n$ at a nodal cubic of dimension $n - 2$
and contracting the proper transform of a hyperplane, which gives
a \emph{singular} $n$\,-\,fold $Y$. The only singularity of $Y$,
besides the node, happens to be of the form $\com^n\slash\mu_2$
and so an appropriate double covering $X \map Y$ makes ($\Gamma$
Cartier and) $Y$ smooth (see {\ref{subsection:exa-0}} below). Here
$X$ is an index $n - 3$ Fano manifold having $\text{Pic}\,X =
\cel\cdot\mathcal{O}_X(\Gamma)$ (we keep the same notation for the
images of $\Gamma$ on $Y$, $X$, etc.). It remains then to estimate
the function $s(\cdot)$ on $X$ and to show that $X$ indeed
compactifies $\com^n$.

The first issue is resolved in
Corollary~\ref{theorem:small-deg-curves-on-x-cor} by an explicit
computation, where we show that $s(\cdot) = 1$ on $\Gamma$, while
$s(\cdot)\ge 2$ on the complement $X\setminus\Gamma$. In turn, the
second issue (that is $X\setminus\Gamma \simeq\com^n$) reduces to
finding a cubic polynomial $Q$ such that the double cover of
$\com^n$ with ramification in $Q$, i.\,e.
$\text{Spec}\,\com[\com^n][\sqrt{Q}]$, is also isomorphic to
$\com^n$. Theorem~\ref{theorem:main} for the given $X$ now follows
from Corollary~\ref{theorem:-five-fold-not-r-u} by combining the
two mentioned properties of $X$ with
Corollary~\ref{theorem:cor-prop-on-mults}.

\begin{remark}
\label{remark:f-2-hassett-yura} The assumption on $\Gamma$ to be
irreducible is crucial in our approach (cf.
Remark~\ref{remark:why-gamma-irred} below). In fact, the surface
$X := \f_1$ is uniformly rational (as a toric surface) and
compactifies $\com^2$, with $\Gamma$ being the union of the
$(-1)$\,-\,curve $Z$ and a ruling $R$ of the natural projection
$\f_1\map\p^1$. Then one can easily see that $s(o) = 2$ (with
respect to $-K_X = 2Z + 3R$) for any point $o\not\in Z$ (see
e.\,g. \cite[Th\'eor\`eme 1.3]{bru}). Otherwise we have $s(o) = 1$
--- in contradiction with what happens for irreducible $\Gamma$.
Anyhow, $X$ is an \emph{equivariant} compactification of $\com^2$,
and it would be interesting to find out whether all such
compactifications of $\com^n$ are u.\,r. (cf. \cite{fu-hwang-1}
and \cite{hassett-tschinkel}).
\end{remark}

\bigskip

\begin{notation}
All varieties, unless stated otherwise, are defined over the field
$\com$ and are assumed to be normal and projective. We will be
using freely standard notation, notions and facts (although we
recall some of them in the text for convenience) from
\cite{isk-pro}, \cite{kollar-rat-curves} and \cite{kol-mor}.
\end{notation}

\bigskip

\thanks{{\bf Acknowledgments.} I am grateful to C. Birkar, F. Bogomolov, A.\,I. Bondal, S. Galkin, M. Romo,
and J. Ross for their interest and helpful comments. Some parts of
the paper were written during my visits to CIRM, Universit\`a
degli Studi di Trento (Trento, Italy), Cambridge University
(Cambridge, UK) and Courant Institute (New York, US). The work was
supported by World Premier International Research Initiative
(WPI), MEXT, Japan, Grant\,-\,in\,-\,Aid for Scientific Research
(26887009) from Japan Mathematical Society (Kakenhi), and by the
State assignment project FSMG-2023-0013 at the Center for Pure
Mathematics (MIPT).

\section*{}

\bigskip

\section{Beginning of the proof of Theorem~\ref{theorem:main}: an obstruction}
\label{section:proo}

\refstepcounter{equation}
\subsection{}
\label{subsection:proo-0}

Let $X$ be a Fano manifold with $\text{Pic}\,X \simeq \mathbb{Z}$
compactifying $\com^n$. In other words, there exists an affine
open subset $U \subset X$, $U \simeq \com^n$, such that
$\text{Pic}\,X = \cel\cdot\mathcal{O}_X(\Gamma)$ for the boundary
$\Gamma := X \setminus U$. Note that under these assumptions
$\Gamma$ must be an \emph{irreducible} hypersurface (see
\cite[Proposition 2.2]{van}).

Fix one particular such $X\ne\p^n$ (see Section~\ref{section:exa}
below for some examples). Let $H$ be a generator of
$\text{Pic}\,X$ and $x_1,\ldots,x_n$ be affine coordinates on $U$.
Then, for some (minimal) $r$, there exist sections $s_i\in
H^0(X,H^r)$ such that $s_i = x_i$ on $U$. Indeed, with $H^r$
\emph{very ample}, $s_i\big\vert_U$ induce an identification $U =
\com^n$.

Now pick a point $p\in\Gamma$ and a rational function
$t\in\mathcal{O}_{X,p} \subset \com(U)$ defining $\Gamma$ in an
affine neighborhood $U ' \subset X$ of $p$.

\begin{lemma}
\label{theorem:t-inverse-is-regular-on-u} We have
$t^{-1}\in\com[U]$ for an appropriate $t$. More precisely,
$t^{-1}\big\vert_U$ is an irreducible polynomial in
$x_1,\ldots,x_n$.
\end{lemma}

\begin{proof}
The function $t$ can be chosen to have no zeroes on $U =
X\setminus\Gamma$ for $\Gamma$ being irreducible. Hence $t^{-1}$
is a polynomial $\in\com[U]$. Its irreducibility follows from that
of $\Gamma$.
\end{proof}

\refstepcounter{equation}
\subsection{}
\label{subsection:int-2}

Now suppose that $X$ \emph{is} uniformly rational. Let $U'\ni p$
be as above. Then $U'$ embeds into $\com^n$.

Let $s\in H^0(X,H)$ be the section whose zero locus equals
$\Gamma$. By definition of $H$ we have $s\big\vert_{U'} = t$ and
$s\big\vert_U = 1$ (cf.
Lemma~\ref{theorem:t-inverse-is-regular-on-u}), so that the
functions $s\big\vert_{U'},s\big\vert_U\in\com(x_1,\ldots,x_n)$
are identified on $U' \cap U$ via $s\big\vert_{U'} =
ts\big\vert_U$. This yields
\begin{equation}
\label{some-eq-1} y_i := s_i\big\vert_{U'} =
t^rx_i\in\mathcal{O}_{X,p}
\end{equation}
for all $i$. Indeed, both line bundles $H^r\big\vert_{U'}$ and
$H^r\big\vert_U$ are trivial on $U'$ and $U$, respectively, for
$s_i\big\vert_{U'}$ and $s_i\big\vert_U$ regarded as rational
functions on $\com^n$, satisfying $s_i\big\vert_{U'}\in\com[U']$
and $s_i\big\vert_U = x_i$ by construction. Then
$H^r\big\vert_{U'}$ and $H^r\big\vert_U$ are glued over $U' \cap
U$ via the multiplication by $t^r$ as \eqref{some-eq-1} indicates.

\begin{lemma}
\label{theorem:y-i-non-const} In the previous setting, if
$y_i\ne\text{const}$ for all $i$, then $y_1,\ldots,y_n$ are local
parameters on $U'\subseteq\com^n$ generating the maximal ideal of
the $\com$\,-\,algebra $\mathcal{O}_{X,p}$.
\end{lemma}

\begin{proof}
Notice that
$$\com(x_1,\ldots,x_n)=\com(U)=\com(U')=\com(y_1,\ldots,y_n)$$
by construction, i.\,e. $x_i=y_i/t^r$ (respectively, $y_i$) are
(birational) coordinates on $U'$, defined everywhere out of
$\Gamma$ (respectively, everywhere on $U'$). This implies that the
morphism $\xi: \com^n \cap (t^{-1} \ne 0) \map (U'\subseteq
\com^n)$, given by
$$
(x_1,\ldots,x_n)\mapsto(y_1 = x_1t^r,\ldots,y_n = x_nt^r),
$$
is birational.\footnote{~More specifically, dividing all the $x_i$
by $x_1$, say, one may assume $x_i=y_i$ on $U'\subseteq\com^n$ for
all $i\ge 2$. Then $\xi$ is simply the multiplication of $x_1$ by
$t^r$.}

Functions $y_i$ do not have common codimension $1$ zero locus on
$U$ (cf. Lemma~\ref{theorem:t-inverse-is-regular-on-u}). Hence
$\xi$ does not contract any divisors. In particular, $\xi^{-1}$ is
well\,-\,defined near $\Gamma$ by Hartogs, which shows that
$y_1,\ldots,y_n$ are the claimed local parameters.
\end{proof}

\begin{lemma}
\label{theorem:phi-const-r-1} $y_i = \text{const}$ for at most one
$i$.
\end{lemma}

\begin{proof}
Indeed, otherwise \eqref{some-eq-1} gives $s_i=s_j$ on $X$ for
some $i\ne j$, a contradiction.
\end{proof}

\begin{lemma}
\label{theorem:phi-const-r-1-a} Let $y_1=\text{const}$. Then
$t^r,y_2,\ldots,y_n\in\mathcal{O}_{X,p}$ are local parameters on
$U'\subseteq\com^n$ generating the maximal ideal of the
$\com$\,-\,algebra $\mathcal{O}_{X,p}$.
\end{lemma}

\begin{proof}
One may assume that $y_1=1$. Then $y_i\ne\text{const}$ for all
$i\ge 2$ by Lemma~\ref{theorem:phi-const-r-1}, and a similar
argument as in the proof of Lemma~\ref{theorem:y-i-non-const}
shows that birational morphism $\eta: \com^n\cap (t^{-1} \ne 0)
\map (U'\subseteq \com^n)$, given by
$$
\eta:
(x_1,\ldots,x_n)\mapsto(t^r=1/x_1,y_2=x_2t^r,\ldots,y_n=x_nt^r),
$$
does not contract any divisors. Hence again $t^r,y_2,\ldots,y_n$
are the asserted local parameters.
\end{proof}

\begin{remark}
\label{remark:why-gamma-irred} An upshot of the previous
considerations is that the whole ``\,analysis\,'' on $X$, encoded
in the line bundle $H$, can be captured by just \emph{two affine}
charts, like $U$ and $U'$, with a transparent gluing (given by
$t$) on the overlap $U \cap U'$. Let us stress one more time that
this holds under the assumption that $X$ \emph{is uniformly
rational}. In addition, as will be also seen in
{\ref{subsection:int-3}}, similar property does not extend
directly to the case of $X$ with reducible boundary $\Gamma$
(compare Proposition~\ref{theorem:prop-on-mults} and
Remark~\ref{remark:f-2-hassett-yura}).
\end{remark}

\refstepcounter{equation}
\subsection{}
\label{subsection:int-3}

Let $h\in H^0(X,H^r)$ be any section. One may write (cf.
{\ref{subsection:int-2}})
\begin{equation}
\label{eq-for-h-u} \frac{h\big\vert_{U}}{s^r\big\vert_{U}} =
\sum_{0 \le i_1 + \ldots + i_n\le m}
a_{i_1,\ldots,i_n}x_1^{i_1}\ldots x_n^{i_n},
\end{equation}
where $a_{i_1,\ldots,i_n}\in\com$, $m = m(h)\ge 0$ and $i_j$ are
non\,-\,negative integers. Now, it follows from \eqref{some-eq-1}
and Lemmas~\ref{theorem:y-i-non-const},
\ref{theorem:phi-const-r-1-a} that
\begin{equation}
\label{eq-for-h-u-prime} h\big\vert_{U'} = \sum_{0 \le i_1 +
\ldots + i_n\le m} a_{i_1,\ldots,i_n}y_1^{i_1}\ldots y_n^{i_n}
t^{m' + r - r(i_1 + \ldots + i_n)}
\end{equation}
for some $m' = m'(h) \ge 0$, by shrinking $U'$ if
necessary.\footnote{~Note that for each $U' \supset \Gamma$ there
is a \emph{specific} trivialization of the line bundle
$\mathcal{O}_X(H)$ --- that is an identification
$\mathcal{O}_X(H)\big\vert_{U'} = (1/t^{m'})\,\mathcal{O}_{U'}$
for some $m'$. One must also have $m' \ge 1$ for general $h$
because $X \ne \p^n$ by assumption (model example here is a smooth
quadric $X \subset \p^4$).} Conversely, starting with \emph{any}
function on $U'$ as in \eqref{eq-for-h-u-prime}, we can find $h
\in H^0(X,H^r)$ such that $h\big\vert_{U'} = \text{the right hand
side of}$ \eqref{eq-for-h-u-prime} (cf.
Remark~\ref{remark:why-gamma-irred}). Indeed, in this way we get a
global section of $H^r$, regular away the codimension $\ge 2$
locus $X \setminus U \cup U'$ (recall that $\Gamma$ is
irreducible), hence regular on the entire $X$.

This discussion condensates to the next

\begin{prop}
\label{theorem:prop-on-mults} There exist a point $o\in U$ and a
point $p = p(o)\in\Gamma\cap U'$ such that for any hypersurface
$\Sigma\sim r\Gamma$,\footnote{~``\,$\sim$\,'' denotes the linear
equivalence of divisors on $X$.} having prescribed multiplicity
$\mathrm{mult}_{o}\,\Sigma > 0$ at $o$, there is a hypersurface
$\hat{\Sigma}\sim r\Gamma$ such that
$\mathrm{mult}_{p}\,\hat{\Sigma} \ge \mathrm{mult}_{o}\,\Sigma$.
Furthermore, the same applies to any $\Sigma\sim k\Gamma$, where
$k \ge r$ is an arbitrary number.
\end{prop}

\begin{proof}
Set $o \in U = \com^n$ to be the origin with respect to the
coordinates $x_i$.

\begin{lemma}
\label{theorem:s-i-intersect-on-gamma} The loci $H_i := (s_i =
0)$, $1\le i\le n$, have a common intersection point, denoted $p$,
on $\Gamma$.
\end{lemma}

\begin{proof}
Let us assume first that $y_j = \text{const}$ for some $j$. Then
we get $r\Gamma = (s_j = 0)$ (see \eqref{some-eq-1}) and hence the
identity
$$
H_1 \cdot \ldots \cdot H_n = \prod_{i=1}^n H_i = \big(\prod_{1 \le
i\ne j \le n} H_i\big)\big\vert_{r\Gamma} = r\deg\Gamma > 0
$$
for the product of cycles. This shows in particular that the
(set\,-\,theoretic) intersection $\displaystyle\bigcap_{i=1}^n H_i
\subset \Gamma$ is not empty.

Assume next all $y_i \ne \text{const}$ and suppose that
$\displaystyle\bigcap_{i=1}^n H_i \cap \Gamma = \emptyset$. Then
$\displaystyle\bigcap_{i=1}^n H_i$ is the (reduced) point $o$,
which immediately gives $X = \p^n$ (recall that $H^r$ is very
ample according to the setting in {\ref{subsection:proo-0}}), a
contradiction.
\end{proof}

Let the section $h \in H^0(X,H^r)$ have $\Sigma$ as its zero
locus. We may assume without loss of generality all but one
$a_{i_1,\ldots,i_n}$ in \eqref{eq-for-h-u} and
\eqref{eq-for-h-u-prime} to be zero. Let also $p$ be as in
Lemma~\ref{theorem:s-i-intersect-on-gamma}. Then, since $t(p) = 0$
and $m' + r - r(i_1 + \ldots + i_n) \ge 0$ in
\eqref{eq-for-h-u-prime} by construction, one may take
$\hat{\Sigma} := \Sigma$ whenever all $y_i \ne \text{const}$.

Finally, if $y_1 = 1$ (and $i_1 \ne 0$), say, then from
\eqref{some-eq-1}, \eqref{eq-for-h-u-prime} and
Lemmas~\ref{theorem:phi-const-r-1}, \ref{theorem:phi-const-r-1-a}
we obtain
$$
\mathrm{mult}_{o}\,\Sigma = i_1\mathrm{mult}_{o}\,t^{-r} + i_2 +
\ldots + i_n = i_1 + i_2 + \ldots + i_n \le \frac{m'}{r} + 1.
$$
It thus suffices to take any $\hat{\Sigma} \ni p$ with equation
$y_2t^{m'} = 0$ on $U'$ (cf. the discussion at the beginning of
{\ref{subsection:int-3}}) for $p$ being given by $t = y_2 = \ldots
= y_n = 0$ and
$$
\mathrm{mult}_{p}\,\hat{\Sigma} \ge m' + 1 \ge
\mathrm{mult}_{o}\,\Sigma.
$$

For the last part of Proposition~\ref{theorem:prop-on-mults}, we
can formally put $r := k$ in the preceding considerations, which
completes the proof.
\end{proof}

\begin{remark}
\label{remark:con-from-proof} The above argument shows that both
divisors $\Sigma$ and $\hat{\Sigma}$ can actually be taken to vary
in some linear systems, having \emph{isolated} base loci near $o$
and $p$, respectively. Furthermore, the value $s(o)$ can be
computed via various linear systems $\mathcal{M} \ni \Sigma$, $k
\ge r$ (cf. {\ref{subsection:intro-1}}), and similarly the value
$s(p)$ can be computed via various linear system containing
$\hat{\Sigma}$.
\end{remark}

\begin{cor}
\label{theorem:cor-prop-on-mults} For $o \in U$ and $p \in \Gamma$
as above we have $s(p) \ge s(o)$.
\end{cor}

\begin{proof}
Fix some $\Sigma \in\mathcal{M}$ and $k := k_i$ as in
\eqref{sesh-const}. One may assume that $\displaystyle\limsup_{i
\to \infty}\,(\mathrm{mult}_{o}\,\Sigma)\slash k$ exists and
equals $s(o)$.

By Proposition~\ref{theorem:prop-on-mults} there is $\hat{\Sigma}
\sim \Sigma$ such that $\mathrm{mult}_{p}\,\hat{\Sigma} \ge
\mathrm{mult}_{o}\,\Sigma$. Then from \eqref{sesh-const} and
Remark~\ref{remark:con-from-proof} we get $s(p)\ge s(o)$ as
wanted.
\end{proof}

With Corollary~\ref{theorem:cor-prop-on-mults} we conclude our
construction of necessary condition for the manifold $X \supset
\com^n$ in {\ref{subsection:proo-0}} to be u.\,r. Let us now
construct those $X$ that do not pass through this simple
obstruction.

\section*{}

\bigskip

\section{End of the proof of Theorem~\ref{theorem:main}}
\label{section:exa}

\refstepcounter{equation}
\subsection{}
\label{subsection:exa-0}

Take the projective space $\p := \p^n, n \ge 4$, with a hyperplane
$H\subset\p$ and a cubic hypersurface $S \subset H$. Let $\sigma:
V \map \p$ be a blowup of (the ideal defining) $S$. More
specifically, for the reasons that will become clear in
{\ref{subsection:exa-1}} below, we assume $V \subset \p^1 \times
\p$ to be given by the (local) equation
$$
wt_0 = (wx_1^2 + F)t_1,
$$
where $t_i$ are projective coordinates on the first factor and
$H,S$ are given by $w = 0,w = F(x_1,\ldots,x_n) = 0$,
respectively, in projective coordinates $w,x_i$ on $\p$.
Furthermore, we take $F$ in the form
$$
x_1x_3^2 + x_2^2x_4 + x_3^3 + F_3,
$$
with a general homogeneous cubic $F_3\in\com[x_4,\ldots,x_n]$. It
is easy to see by Bertini theorem that the only singularity of $V$
is at the point $O := [1:1] \times [0:1:0:\ldots:0]$. We will
identify the analytic germ $(O \in V)$ with
\begin{equation}
\label{sing-germ} 0 \in (wt = x_3^2 + \phi(x_2,x_4,\ldots,x_n))
\subset \com^{n+1}
\end{equation}
for some holomorphic function $\phi$ not depending on $x_3$.
Indeed, this is possible after one replaces $x_3$ with $x'_3 :=
x_3\sqrt{1+x_3}$, locally near the origin $0 \in \com^{n+1}$.

Put $E := \sigma^{-1}S$ and $H^* := \sigma^*H$. Notice that
$\sigma$ resolves the indeterminacies of the linear system $|3H -
S|$. Let $\varphi:V\map Y$ be the corresponding morphism onto some
variety $Y$ with very ample divisor $\mathcal{O}_Y(1)$ pulling
back to $3H^* - E$.

\begin{lemma}
\label{theorem:phi-is-bir-and-h} $\varphi$ is birational and
contracts the divisor $H_V := \sigma_*^{-1}H\simeq\p^{n-1}$ to a
point.
\end{lemma}

\begin{proof}
Consider the affine open subset $\com^n = \p \cap (w \ne 0)$ with
coordinates $x_1,\ldots,x_n$. Note that the restriction $|3H -
S|\big\vert_{\com^n}$ is spanned by $1$, $F$, $x_i$ and the
quadrics $x_ix_j,1 \le i,j \le n$, since the degree $3$ part of
the ideal of the cubic $S$ is generated by $w^3$, $F$ and various
$w^2x_i$, $wx_ix_j$. This (obviously) implies that
$\varphi\big\vert_{\com^n}$, hence $\varphi$ as well, is
birational.

Now let $Z\subset\p$ be the image of a curve contracted by
$\varphi$. Suppose that $Z\not\subset H$. Then, since the proper
transform $\sigma_*^{-1}Z$ is contracted by $\varphi$, we have
$3H\cdot Z = \deg S\cdot Z$. On the other hand, we have $\deg
S\cdot Z\le H\cdot Z$ because $Z\not\subset H$, a contradiction.
Thus every curve contracted by $\varphi$ belongs to
$H_V\simeq\p^{n-1}$.
\end{proof}

Note that $O \not\in H_V$ by construction. Then $(\varphi(O) \in
Y) \simeq (O \in V)$ and the variety $Y$ can have at most one more
singular point (cf. Lemma~\ref{theorem:c-4-z-2-sing} below).
Specifically, $\varphi\circ\sigma^{-1}$ induces an isomorphism
between $\p\setminus H\simeq\com^n$ and $Y\setminus\varphi(E)$, so
that $Y$ can also be singular at the point $o := \varphi(H_V)$ on
the boundary $\varphi(E)$ (cf.
Lemma~\ref{theorem:phi-is-bir-and-h}).

Further, we want to modify $Y$ into a \emph{smooth} $n$\,-\,fold
(our pertinent $X$), yet preserving the properties $\com^n\subset
Y$ and $\text{Pic}\,Y = \cel$. Let us start with the following
technical observation:

\begin{lemma}
\label{theorem:c-4-z-2-sing} The singularity germ $(o\in Y)$ is
locally analytically isomorphic to the quotient
$\com^n\slash\mu_2$ for the $2$\,-\,cyclic group $\mu_2$ acting
diagonally on $\com^n$.
\end{lemma}

\begin{proof}
Recall that $K_{V} = -(n+1)H^* + E$ and $\varphi$ contracts $H_V =
\sigma_*^{-1}H \sim H^* - E$ to the point $o$.

One can choose (sufficiently general) divisors $D_1,\ldots,D_n$ on
$Y$ such that $\varphi_*^{-1}D_i\sim H^*$ for all $i$ and the pair
$$(V,\sum_{i=1}^n \varphi_*^{-1}D_i +
H_V)$$ is log canonical. Note also that
$$K_{V} + \sum_{i=1}^n \varphi_*^{-1}D_i +
H_V = 0.$$ Then we apply \cite[Lemma 3.38]{kol-mor} to deduce that
$(Y,\displaystyle\sum_{i=1}^n D_i)$ is log canonical because $K_Y
+ \displaystyle\sum_{i=1}^n D_i = 0$ by construction. Now let us
run standard argument showing that locally analytically near $o$
the pair $(Y,\displaystyle\sum_{i=1}^n D_i)$ is \emph{toric}.
Namely, one may assume all $D_i$ to be $\ra$\,-\,Cartier, in which
case \cite[18.22\,--\,18.23]{kollar-92} implies that $o \in Y$ is
a toric singularity.

In particular, $(o \in Y)$ is of the form $\com^n\slash\mu_m$ for
a cyclic group $\mu_m$ acting diagonally on $\com^n$, and it
remains to show that $m = 2$. For the latter, notice that
$\sigma(\varphi_*^{-1}D_i)$ are hyperplanes on $\p$, with the
plane
$\sigma(\varphi_*^{-1}D_1)\cap\ldots\cap\sigma(\varphi_*^{-1}D_{n-2})$,
say, intersecting the cubic $S$ at exactly $3$ distinct points.
These points belong to the line $H \cap
\sigma(\varphi_*^{-1}D_1)\cap\ldots\cap\sigma(\varphi_*^{-1}D_{n-2})$
and the morphism $\sigma$ blows them up. This implies that $H_V
\cap \varphi_*^{-1}D_1\cap\ldots\cap \varphi_*^{-1}D_{n-2}$ is a
$(-2)$\,-\,curve on the smooth surface
$\varphi_*^{-1}D_1\cap\ldots\cap \varphi_*^{-1}D_{n-2}$ and the
equality $m = 2$ follows by varying $D_i$.
\end{proof}

Choose some generic hypersurface $R \in |3(3H^* - E)|$ such that
$R \cap H_V = \emptyset$ (cf.
Lemma~\ref{theorem:phi-is-bir-and-h}) and locally analytically
near $O$ the divisor $R \subset V$ is obtained by intersecting
\eqref{sing-germ} with $(x_3 = 0)$. Let $\pi:\tilde{V} \map V$ be
the double covering ramified in $R + H_V \sim 10H^* - 4E$. Then
the variety $\tilde{V}$ is \emph{smooth} because locally over the
point $O$ it is isomorphic to the germ
$$
0 \in (wt = x_3 + \phi(x_2,x_4,\ldots,x_n)) \subset \com^{n+1}.
$$
Furthermore, we have
$$
-K_{\tilde{V}} = -\pi^*(K_V + \frac{1}{2}(R + H_V)) =
\pi^*((n-4)H^* + E) := (n-4)\tilde{H} + \tilde{E}
$$
by Hurwitz formula, where $\tilde{H}$ and $\tilde{E}$ are the
pullbacks to $\tilde{V}$ of $H^*$ and $E$, respectively.

It is immediate from the construction that the group
$\mathrm{Pic}\,\tilde{V}$ is generated by
$\mathcal{O}_{\tilde{V}}(\pi^{-1}H_V)$ and
$\mathcal{O}_{\tilde{V}}(\tilde{E})$ (note that $\pi^*H_V =
2\pi^{-1}H_V$ because $\pi$ ramifies over $H_V$). Indeed, since
$\mathcal{O}_V(H_V)$ and $\mathcal{O}_V(E)$ generate
$\mathrm{Pic}\,V$, with intersections $E \cap H_V$, $E \cap R$
being reduced and $E$, $H_V$ being \emph{contractible} divisors,
the line bundles $\mathcal{O}_{\tilde{V}}(\pi^{-1}H_V)$ and
$\mathcal{O}_{\tilde{V}}(\tilde{E})$ are the claimed generators of
$\mathrm{Pic}\,\tilde{V}$.

\begin{lemma}
\label{theorem:x-is-smooth} There exists a birational contraction
$f: \tilde{V} \map X$ of $\pi^{-1}H_V$, given by a multiple of the
linear system $|\pi^*(3H^* - E)|$, onto some \emph{smooth} variety
$X$.
\end{lemma}

\begin{proof}
Let $Z\subset\pi^{-1}H_V\simeq\p^{n-1}$ be a line. We have
$$
K_{\tilde{V}} \cdot Z = -((n-4)\tilde{H} + \tilde{E}) \cdot Z = 3
- n < 0.
$$
Then \cite[Theorem 3.25]{kol-mor} delivers the contraction $f$ as
stated. Finally, Lemma~\ref{theorem:c-4-z-2-sing} yields
$$
\pi^{-1}H_V\cdot Z = \pi^*H_V\cdot Z = \frac{1}{2}H_V \cdot \pi(Z)
= -1,
$$
which implies that $f$ is just the blowup of the smooth point
$f(\pi^{-1}H_V) \in X$.
\end{proof}

It follows from Lemma~\ref{theorem:x-is-smooth} that $X$ is a
smooth Fano $n$\,-\,fold of index $n - 3$. Namely, we have
$$
-K_X = (n - 3)f_*\tilde{H} = (n - 3)f_*\tilde{E},
$$
for $\mathrm{Pic}\,X = \cel \cdot \mathcal{O}_X(f_*\tilde{E})$.

Let us now find those curves on $X$ having the smallest
intersection number with $f_*\tilde{E}$:

\begin{prop}
\label{theorem:small-deg-curves-on-x} For every curve $Z \subset
X$ we have $f_*\tilde{E} \cdot Z \ge 1$ and the equality is
achieved when $\sigma(\pi(f_*^{-1}Z))$ is a point on $\p$. In
other words, $f_*^{-1}Z \subset \tilde{E}$ is an elliptic curve,
which is the preimage of a ruling on $E$.
\end{prop}

\begin{proof}
Notice first that
$$
-K_{\tilde{V}} = \frac{n-3}{2}(3\tilde{H} - \tilde{E}) -
\frac{n-1}{2}(\tilde{H} - \tilde{E}).
$$
In particular, we get $f^*(-K_X) =
\displaystyle\frac{n-3}{2}(3\tilde{H} - \tilde{E})$, and hence
$f_*\tilde{E} \cdot Z = a$ if and only if $-K_X \cdot Z = (n-3)a$
if and only if
$$
\frac{n-3}{2}(3\tilde{H} - \tilde{E}) \cdot f_*^{-1}Z = (n-3)a
$$
for any $a \in \cel$.

Further, if $\pi(f_*^{-1}Z)$ is a ruling on $E$, then $\tilde{H}
\cdot f_*^{-1}Z = 0$ by definition and
$$
\tilde{E} \cdot f_*^{-1}Z = \pi^*E \cdot f_*^{-1}Z = E \cdot
\pi_*(f_*^{-1}Z) = 2E \cdot \pi(f_*^{-1}Z) = -2
$$
by the projection formula, where $\pi(f_*^{-1}Z)$ has intersection
index $4$ with the ramification divisor $R + H_V$, i.\,e.
$f_*^{-1}Z$ is an elliptic curve. This implies that $a = 1$ for
such $Z$ and Proposition~\ref{theorem:small-deg-curves-on-x}
follows.
\end{proof}

\begin{cor}
\label{theorem:small-deg-curves-on-x-cor} For every point $p \in
X$ we have $s(p) = 1$ when $p \in f_*(\tilde{E})$ and $s(p)\ge 2$
otherwise.
\end{cor}

\begin{proof}
Consider the case when $p = f(\pi^{-1}H_V)\in f_*\tilde{E}$ first.
Note that the Mori cone $\overline{NE}(\tilde{V}) \subset
N_1(\tilde{V})\otimes\mathbb{R} = \mathbb{R}^2$ is generated by
the classes of a line in $\pi^{-1}H_V\simeq\p^{n-1}$ and an
elliptic curve $Z \subset \tilde{E}$ as in
Proposition~\ref{theorem:small-deg-curves-on-x}. Now, by
construction of $\tilde{V}$ via the blowup $f$ of $X$ at $p$ (see
Lemma~\ref{theorem:x-is-smooth}) we obtain that $s(p) = 1$, since
the divisor $\tilde{H}$ is nef and
$$
f^*f_*\tilde{E} - \lambda\pi^{-1}H_V = (\frac{3}{2} -
\frac{\lambda}{2})\tilde{H} + \frac{1}{2}(\lambda - 1)\tilde{E}
$$
is nef only when $\lambda \le \displaystyle 1$. Then the estimate
$s(p) \ge 1$ holds for any other $p\in f_*\tilde{E}$ due to the
lower semi\,-\,continuity of the function $s(\cdot)$ on $X$ (see
\cite[Example 5.1.11]{laz-pos-in-ag-I}). But the situation $s(p) >
1$ can not occur for these $p$ because otherwise the divisor
$\beta^*f_*\tilde{E} - \lambda \mathcal{E}$ (we are using the
notation of {\ref{subsection:intro-1}}), with $\lambda
> 1$, intersects the curve $\sigma_*^{-1}Z$ as $1 - \lambda < 0$.
Thus $s(\cdot) = 1$ identically on $f_*\tilde{E}$.

Further, recall that $\pi$, when considered on $\tilde{V}
\setminus \pi^{-1}H_V \cup \tilde{E} = X \setminus f_*\tilde{E}$,
is the double cover of $V \setminus H_V \cup E\simeq\com^n$
ramified in $R\setminus R \cap E$. Also, the proper transform on
$\tilde{V}$ of any element $\Sigma\in |mf_*\tilde{E}|,m\in\cel$,
is an element from $|\displaystyle\frac{m}{2}(3\tilde{H} -
\tilde{E})|$, which is mapped (via $\sigma\circ\pi$) onto some
divisor $\Sigma' \in |m(3H - S)|$ on $\p$. In particular, we get
\begin{equation}
\label{mult-est-sigma} \mathrm{mult}_p\,\Sigma =
\mathrm{mult}_{\sigma\circ\pi(p)}\,\Sigma' \qquad \text{or}\qquad
\ge \mathrm{mult}_{\sigma\circ\pi(p)}\,\Sigma'
\end{equation}
as long as $p\not\in f_*\tilde{E}$ (for $p\in X$ identified with
$f^{-1}(p)\in \tilde{V}$), depending on whether $p\not\in R$ or
$p\in R$, respectively.

Now take $m = 1$ and $\Sigma'\in |3H - S|$ satisfying
$\mathrm{mult}_{\sigma\circ\pi(p)}\,\Sigma' = 2$. Such $\Sigma'$
varies in a linear system on $\p$ with isolated base locus near
$\sigma\circ\pi(p)$.\footnote{~Indeed, if $w, x_1,\ldots,x_n$ are
projective coordinates on $\p$ as in {\ref{subsection:exa-0}},
then we consider $\Sigma' := (F + wB = 0)$ for an arbitrary
quadratic form $B = B(x_1,\ldots,x_n)$ and $\sigma\circ\pi(p) :=
[1:0:\ldots:0]$. The case of any other point in $\com^n = \p \cap
(w \ne 0)$ is easily reduced to this one.} This and
\eqref{mult-est-sigma} (cf. \eqref{sesh-const}) imply that $s(p)
\ge 2$ for $f_*\tilde{E} \sim \Sigma$ also varying in in a linear
system on $X$ with isolated base locus near $p$.
\end{proof}

\refstepcounter{equation}
\subsection{}
\label{subsection:exa-1}

It remains to show that $X\setminus f_*\tilde{E}\simeq\com^n$ for
one particular $R$.

Identifying $\p \setminus H = V \setminus H_V \cup E = Y \setminus
\varphi(E)$ with $\com^n = \p \cap (w \ne 0)$ via $\sigma,\varphi$
we observe that there are elements $y_1,\ldots,y_n$ in
$\displaystyle\bigoplus_{m \ge 0}H^0(V,m(3H^* -
E))\big\vert_{\com^n}$, depending on the affine coordinates $x_i$,
for which the assignment $x_i \mapsto y_i$, $1\le i\le n$, induces
an \emph{automorphism} on $\com^n =
\varphi\circ\sigma^{-1}(\com^n)$. Namely,
$$
y_1 := x_1 + F,\ y_2 := x_2,\ y_3 := x_2x_3,\ \ldots, \ y_n :=
x_2x_n
$$
satisfy this property, since one has an isomorphism
$\com(y_1,\ldots,y_n) \simeq \com(x_1,\ldots,x_n)$ by the choice
of $F$ in {\ref{subsection:exa-0}}. This also shows that one may
assume $y_i\big\vert_E \ne 0$ identically for all $i$.

Further, the equation of $R$ on $V\setminus H_V \cup E$ is a cubic
polynomial in $y_i$, and we may take
$$
R \cap (V \setminus H_V \cup E) := (P(y_1,\ldots,y_{n-1}) + y_n +
1 = 0)
$$
for some general homogeneous $P$ such that the local analytic
condition on $R$ (right after Lemma~\ref{theorem:c-4-z-2-sing}) is
met. Notice that this defines a \emph{smooth} hypersurface in
$\com^n$.

Expressing $y_i$ in terms of $x_i$ we identify $R \cap (V
\setminus H_V \cup E)$ with a hypersurface in $\p\setminus H$.
Then compactifying via $w$, we obtain that $R \subset V$ can only
be singular at the locus $y_1 = \ldots = y_{n-1} = w = 0$, i.\,e.
precisely at $S$.

\begin{lemma}
\label{theorem:r-is-smooth} $R$ is smooth and $R \cap H_V =
\emptyset$.
\end{lemma}

\begin{proof}
After the blowup $\sigma$ the only singularities on $E$ that $R =
(P + w^7y_n + w^9 = 0)$ can have belong to the locus $I :=
E\cap\displaystyle\bigcap_{i=1}^{n-1}(y_i = 0)\cap H_V$. But the
latter is empty because $y_i\big\vert_E \ne 0$ identically, $1 \le
i \le n - 1$, so that the intersection $I$ has codimension $-1$.
Hence $R$ is smooth. The claim about $R \cap H_V$ follows from
Lemma~\ref{theorem:phi-is-bir-and-h} and again the fact that
$y_i\big\vert_E \ne 0$ identically for $y_i \in |3H^* - E|$.
\end{proof}

Lemma~\ref{theorem:r-is-smooth} implies that $\tilde{V}$ is
smooth. Then so is $X$ (cf. Lemma~\ref{theorem:x-is-smooth}) and
on the \emph{affine} open chart $\tilde{V} \setminus \pi^{-1}H_V
\cup \tilde{E} = X \setminus f_*\tilde{E}$ morphism $\pi:
\tilde{V} \map V$ coincides with the projection of
$$
\tilde{V} \setminus \pi^{-1}H_V \cup \tilde{E} = (T^2 =
P(y_1,\ldots,y_{n-1}) + y_n + 1) \subset \com^{n+1}
$$
onto $\com^n = V \setminus H_V \cup E$ with coordinates $y_i$.
This yields $\tilde{V} \setminus \pi^{-1}H_V \cup
\tilde{E}\simeq\com^n$, if one takes $T,y_1,\ldots,y_{n-1}$ as
generators of the affine algebra $\com[\tilde{V} \setminus
\pi^{-1}H_V \cup \tilde{E}]$.

Finally, since the defining equation of $R\big\vert_E$ is
$P(y_1,\ldots,y_{n-1}) = 0$, with general $P$, both cycles
$R\big\vert_E$ and $\tilde{E}$ are irreducible. Thus
$f_*\tilde{E}$ is also irreducible and so $X$ is the required
compactification of $\com^n$ (with the boundary divisor $\Gamma =
f_*\tilde{E}$). This completes the construction of $X$.

Theorem~\ref{theorem:main} now follows from the next

\begin{cor}
\label{theorem:-five-fold-not-r-u} $X$ is not uniformly rational.
\end{cor}

\begin{proof}
Notice first that $X$ satisfies all the assumptions of
Section~\ref{section:proo}. Then
Corollary~\ref{theorem:cor-prop-on-mults} applies to $X$ once we
assume the latter to be u.\,r. Hence we must have $s(p) \ge s(o)$
for some $p\in f_*\tilde{E}$ and $o \in X\setminus f_*\tilde{E}$.
At the same time,
Corollary~\ref{theorem:small-deg-curves-on-x-cor} gives $s(p) = 1$
and $s(o) \ge 2$, a contradiction. Hence $X$ can not be u.\,r.
\end{proof}


\bigskip

\begin{thebibliography}{16}

\bibitem{bodn}
G. Bodn\'{a}r\ et al., Plain varieties, Bull. Lond. Math. Soc.
{\bf 40} (2008), no.~6, 965 -- 971.

\smallskip

\bibitem{bog-boh}
F. Bogomolov\ and\ C. B\"ohning, On uniformly rational varieties,
in {\it Topology, geometry, integrable systems, and mathematical
physics}, 33 -- 48, Amer. Math. Soc. Transl. Ser. 2, 234, Amer.
Math. Soc., Providence, RI.

\smallskip

\bibitem{bru}
A. Broustet, Constantes de Seshadri du diviseur anticanonique des
surfaces de del Pezzo, Enseign. Math. (2) {\bf 52} (2006),
no.~3\,-\,4, 231 -- 238.

\smallskip

\bibitem{dem}
J.\,-\,P. Demailly, Singular Hermitian metrics on positive line
bundles, in {\it Complex algebraic varieties (Bayreuth, 1990)}, 87
-- 104, Lecture Notes in Math., 1507, Springer, Berlin.

\smallskip

\bibitem{fu-hwang-1}
B. Fu\ and\ J.\,-\,M. Hwang, Uniqueness of equivariant
compactifications of $\Bbb C\sp n$ by a Fano manifold of Picard
number 1, Math. Res. Lett. {\bf 21} (2014), no.~1, 121 -- 125.

\smallskip

\bibitem{gromov-oka-pri}
M. Gromov, Oka's principle for holomorphic sections of elliptic
bundles, J. Amer. Math. Soc. {\bf 2} (1989), no.~4, 851 -- 897.

\smallskip

\bibitem{hassett-tschinkel}
B. Hassett\ and\ Y. Tschinkel, Geometry of equivariant
compactifications of ${\bf G}\sb a\sp n$, Internat. Math. Res.
Notices {\bf 1999}, no.~22, 1211 -- 1230.

\smallskip

\bibitem{isk-pro}
V. A. Iskovskikh\ and\ Yu.\ G. Prokhorov, Fano varieties, in {\it
Algebraic geometry, V}, 1 -- 247, Encyclopaedia Math. Sci., 47,
Springer, Berlin.

\smallskip

\bibitem{around-u-r-II}
I. Karzhemanov, Around the uniform rationality II, {\it in
Preparation}.

\smallskip

\bibitem{kollar-rat-curves}
J. Koll\'ar, {\it Rational curves on algebraic varieties},
Ergebnisse der Mathematik und ihrer Grenzgebiete. 3. Folge. A
Series of Modern Surveys in Mathematics, 32, Springer, Berlin,
1996.

\smallskip

\bibitem{kollar-92}
J. Koll\'ar et al., Flips and abundance for algebraic threefolds,
Ast\`erisque {\bf 211} (1991), 258 p.

\smallskip

\bibitem{kol-mor}
J. Koll\'ar\ and\ S. Mori, {\it Birational geometry of algebraic
varieties}, translated from the 1998 Japanese original, Cambridge
Tracts in Mathematics, 134, Cambridge Univ. Press, Cambridge,
1998.

\smallskip

\bibitem{laz-pos-in-ag-I}
R. Lazarsfeld, {\it Positivity in algebraic geometry. I},
Ergebnisse der Mathematik und ihrer Grenzgebiete. 3. Folge. A
Series of Modern Surveys in Mathematics, 48, Springer, Berlin,
2004.

\smallskip

\bibitem{lie-pti}
A. Liendo\ and\ C. Petitjean, Uniformly rational varieties with
torus action, Transform. Groups {\bf 24} (2019), no.~1, 149 --
153.

\smallskip

\bibitem{pro-coindex-3}
Yu.\ G. Prokhorov, Compactifications of $\com^4$ of index 3,
Algebraic geometry and its applications. Proceedings of the 8th
algebraic geometry conference. Yaroslavl', Russia, August 10--14,
1992; translation in Braunschweig: Viewhweg, Aspects Math. E. {\bf
25} (1994), 159 -- 169.

\smallskip

\bibitem{van}
A. van de Ven, Analytic compactifications of complex homology
cells, Math. Ann. {\bf 147} (1962), 189 -- 204.

\end{thebibliography}
\end{document}